\documentclass{article}
\usepackage{fullpage}
\usepackage{amsmath,amsthm,amssymb}
\newtheorem{theorem}{Theorem}[section]
\newtheorem{lemma}[theorem]{Lemma}
\newtheorem{prop}[theorem]{Proposition}
\newtheorem{cor}[theorem]{Corollary}

\newtheorem{conj}{Conjecture}
\newtheorem{defn}{Definition}

\title{On the sparsity of integers $a$ in solutions to $a!b!=c!$}
\author{Joshua Cooper, Joseph Preuss}
\date{\today}
\begin{document}
\maketitle

\begin{abstract}

We consider the Diophantine equation
$$
a!b! = c!
$$
due to Erd\H{o}s, where we assume $a \leq b$.  It is widely believed that there are only finitely many nontrivial solutions, and considerable work has been dedicated to showing this.  In one direction, Luca (\cite{Luca07}) showed that the set of $c$'s which can appear in solutions has density zero.  Here we show that the set of $a$'s appearing in solutions is also sparse.  In particular, $a$ cannot be one less than a large fraction of primes, and, under the assumption that $\sqrt[k]{a!} \mod 1$ is equidistributed in an appropriate sense, we show that the set of such $a$'s has asymptotic density zero.
\end{abstract}

\section{Introduction}

Erd\H{o}s asks the following problem in \cite{ErdosGraham1980}: for which integers $a_1, \ldots, a_t$ with $t \geq 2$ and $a_i \geq 2$ for each $i$ does there exist an integer $c$ so that $\prod_{i=1}^t a_i! = c!$?  This question, which is open even for $t=2$, has been the inspiration for dozens of papers in elementary number theory since at least the 1970s.  It is widely believed that there are only finitely many ``nontrivial'' solutions, i.e., those for which $c - \max_i a_i > 1$.  In 1991, Erd\H{o}s \cite{Erdos91} showed that for sufficiently large $c$, $k \leq 5 \log\log c$; the constant was subsequently lowered, most recently by \cite{BkhatRamachandra10}, to $(1+o(1))/\log 2 \approx 1.44$. In 2007, Luca \cite{Luca07} showed that there are only finitely many nontrivial solutions if the notorious $abc$ Conjecture holds, and Nair-Shorey \cite{NAIR2016307} refined this to a complete list of solutions conditional on an explicit form of the $abc$ Conjecture due to Baker.  In the same paper, Luca showed that the set of $c$'s that arise in solutions is of asymptotic density zero.  Here, we focus on the $t=2$ case, i.e., $a!b!=c!$, showing that the possible values of $a$ are also rather sparse.

The following definition will play a key role throughout.

\begin{defn}
A solution to $a! b! = c!$ with $a<b$ in the integers is said to be ``class k'' if $c - b = k$.
\end{defn}

Note that class $0$ solutions must have $a=0$ or $a=1$ and $b=c$; it is easy to see that the solution is class $1$ iff $b=a!-1$ and $c=a!$.  Thus, a solution to $a!b!=c!$ is said to be ``trivial'' if $k \leq 1$. It is conjectured that the only non-trivial solution is the triple $(6,7,10)$.  Already the conjecture that there are no class $2$ solutions is open (see \cite{BerendHarmse06}).  In the definition above, and throughout the sequel, we assume that $a < b$.  It is straightforward to see that $a \neq b$ for nontrivial solutions.  Indeed, if $a=b$, then $c!=a!^2$.  However, writing $\nu_p(n!)$ for the $p$-adic order of $n$, this implies that $\nu_p(c!)=2\nu_p(a!)$ for every prime $p$, contradicting that $\nu_p(c!) = 1$ for every prime in $(c/2,c]$, which exists by Bertrand's Postulate.

In Section \ref{sec:primes}, we show that (Corollary \ref{cor:SerreSet}) for any fixed $k>1$, there exists $1/k$ natural density fraction of primes $p$ such that if $a = p-1$, there are no solutions to $a!b!=c!$.  The proof makes use of some interesting elementary number theory: old results of Westland-Fl\"uge on irreducible polynomials which are sparse in the falling-factorial basis, Wilson's Theorem, and Dirichlet's Theorem.  Then, in Section \ref{sec:equidistribution}, by analyzing the asymptotics of falling factorials, we show that $c$ falls into a small interval defined by $a$ and $k$, implying that if $\sqrt[k]{a!} \pmod 1$ is sufficiently equidistributed as $a,k \rightarrow \infty$ -- a much stronger version of which we conjecture is true -- then the set of $a$’s that appear in a solution is also of asymptotic density zero.  Below we write $x^{\underline{k}}$ for the $k$-th falling factorial of $x$, i.e., $x(x-1)\cdots (x-k+1)$.

Here we collect a few useful bounds on possible solutions, especially, that $k < a < b < c$.

\begin{prop} \label{prop:bound_a}
    If $a!b!=c!$ is a solution of class $k$, i.e., $c-b = k$, then $k < a < k + 2 \lceil \log_2 c \rceil$.
\end{prop}
\begin{proof}
    Write $s_q(n)$ for the sum of the base-$q$ digits of $n$. Clearly, $\nu_p(a!)+\nu_p(b!)=\nu_p(c!)$.  Applying Legendre's Formula, we obtain
    $$
    \frac{a-s_p(a)}{p-1} + \frac{b-s_p(b)}{p-1} = \frac{c-s_p(c)}{p-1},
    $$
    from which it follows that $a-(c-b) = s_p(a)+s_p(b)-s_p(c)$ for every prime $p$.
    Since $a < b < c$, in particular
    $$
    | s_2(a)+s_2(b)-s_2(c) | < 2 \lceil \log_2 c \rceil .
    $$
    Thus, $|a-k| < 2 \lfloor \log_2 c \rfloor$.  Now, we argue that $a > k$.  First,
    \begin{align*}
    a! & = \frac{c!}{b!} = c(c-1)\cdots(c-k+1) \\
    & > (c-k+1)^k = (b+1)^k \geq (a+2)^k.
    \end{align*}
    Then taking the log of both sides above and applying Stirling's approximation (that $\log a! < a \log a$) gives
    $$
    a \log a > \log a! > k \log (a+2)
    $$
    so that $k < a \log a / \log(a+2) < a$.
\end{proof}

\section{Solutions Modulo a Prime} \label{sec:primes}

In this section, we consider the equation $a!b!=c!$  for fixed $k$ modulo various primes, showing that there are many values of $a$ for which there can be no solution.  We include the following standard lemma for completeness.

\begin{lemma} \label{lem:-3} For a prime $p \equiv 5 \pmod{6}$, $-3$ is a quadratic nonresidue modulo $p$. 
\end{lemma}
\begin{proof}
    By multiplicativity of the Legendre symbol and Euler's criterion,
    $$
        \left ( \frac{-3}{p} \right ) = \left ( \frac{-1}{p} \right ) \left ( \frac{3}{p} \right ) = (-1)^{(p-1)/2} \left ( \frac{3}{p} \right ).
    $$
    Then, by Quadratic Reciprocity,
    \begin{align*}
    \left ( \frac{-3}{p} \right ) &= (-1)^{(p-1)/2} (-1)^{(p-1)(3-1)/4} \left ( \frac{p}{3} \right ) \\
    &= (-1)^{p-1} \left ( \frac{2}{3} \right ) = -1
    \end{align*}
    since $p \equiv 2 \pmod{3}$ and $2$ is a quadratic nonresidue modulo $3$.
\end{proof}

The preceding lemma can then be used to rule out many values of $a$ as possible participants in a solution to $a!b!=c!$ with $k=2$, namely, $a$ cannot be one less than a prime which is $5 \pmod{6}$.

\begin{prop} There are no class $2$ solutions to $a!b!=c!$ in positive integers if $a+1$ is a prime congruent to $5 \pmod{6}$.
\end{prop}
\begin{proof} 
Note that $c-b=2$ and $a!b!=c!$ implies $a! = c(c-1)$.  Suppose $a!=c(c-1)$ where $a+1=p$ is a prime congruent to $5 \pmod{6}$.  Then
$$
c^2 - c - a! = 0
$$         
so we have $c = \frac{1\pm \sqrt{1-4(-a!)}}{2}$.  Thus, $1+4a! = 1 + 4(p-1)!$ is a perfect square.  By Wilson's Theorem, $1+4(p-1)! \equiv 1+4(-1) = -3 \pmod{p}$, so $-3$ is a quadratic residue modulo $p$.  However, this contradicts Lemma \ref{lem:-3}.
\end{proof}

Next, we consider the case of class $4$: again, half of the primes can be ruled out as possible values of $a+1$.

\begin{prop} \label{thm:pIs4} There are no class $4$ solutions to $a!b!=c!$ in positive integers if $a+1$ is a prime congruent to $2$ or $3 \pmod{5}$.
\end{prop}
\begin{proof}
Suppose $a! = c(c-1)(c-2)(c-3)$, where $a=p-1$ for some prime $p$.  Note that $b>0$ implies $c=b+4>4$, so $c!/b! \geq 24 > 3!$, so we may assume $a>3$ and so $p>4$; furthermore, if $b>0$ then $4! = c(c-1)(c-2)(c-3)$ only has the solution $c=4$, but then $b = c-4 = 0$, a contradiction, so $p \neq 5$. Let $r \equiv c/2 \pmod{p}$, and apply Wilson's Theorem to conclude that
$$
2r(2r-1)(2r-2)(2r-3) \equiv -1 \pmod{p}.
$$
Now, 
$$
2r(2r-1)(2r-2)(2r-3) = (4r^2-6r+1)^2-1,
$$
so 
$$
(4r^2-6r+1)^2 \equiv 0 \pmod{p},
$$
i.e., $p | 4r^2 - 6r + 1$.  However, $(2r-3/2)^2 - 5/4 = 4r^2 - 6r + 1$, so $5/4$ is a quadratic residue mod $p$, which happens iff $5$ is a quadratic residue mod $p$.  It is known this occurs only for primes $p \equiv 1,4 \pmod{5}$.
\end{proof}

For general $k$, we can show that at least a $1/k$ fraction of the primes cannot be $a+1$ if $F_k := x^{\underline{k}}+1$ is irreducible.

\begin{theorem} \label{thm:PositiveDensityIfIrreducible}Suppose $F_k(x) = x(x-1)\cdots (x-k+1)+1$ is irreducible over $\mathbb{Z}$.  Let $P_0$ be the set of primes so that, if $a=p-1$ for some $p \in P_0$, then $a!b!=c!$ has no class $k$ solutions in positive integers.  Then the natural density of $P_0$ in the primes is at least $1/k$.
\end{theorem}
\begin{proof}
    Suppose $a!b!=c!$ is class $k$ and $a = p-1$ for some prime $p$.  Then $c = b+k$ and $a! = c!/b! = F_k(c)-1$.  Then, reducing this equation modulo $p$ yields $c^{\underline{k}} \equiv a! \pmod{p}$, and by Wilson's Theorem, this is equivalent to $F_k(x) \equiv 0 \pmod{p}$.  However, by Theorem 1 and 2 of \cite{Serre03}, the set $P_0$ of primes $p$ for which $F_k(x)$ does not have solution modulo $p$ has positive density, and the density is at least $1/k$.
\end{proof}

\begin{prop}\label{thm:WestlundFluge}
    For $k \geq 2$, $F_k(x)$ is reducible over $\mathbb{Z}$ iff $k=4$.    
\end{prop}
\begin{proof}
    This is an immediate consequence of a result by Westlund and Fl\"{u}ge, answering a question of Schur, as discussed in \cite{DorwartOre33}.
\end{proof}

\begin{cor} \label{cor:SerreSet}
    Fix $k > 1$.  For at least a $1/k$ natural density subset of primes $p$, the equation
    $$
    (p-1)! b! = (b+k)!
    $$
    has no solutions for integers $b$.
\end{cor}
\begin{proof}
    By Theorem \ref{thm:WestlundFluge}, $F_k$ is irreducible if $k \neq 4$, so Theorem \ref{thm:PositiveDensityIfIrreducible} implies that $a!b!=c!$ has no solutions in the integers when $c = b+k$ and $a=p-1$, for an at least $1/k$ density subset of the primes.  Proposition \ref{thm:pIs4} yields the remaining case, by Dirichlet's Theorem on primes in arithmetic progressions.
\end{proof}

From considerable numerical evidence, it appears to be the case that $x^{\underline{k}}-a!$ is irreducible when $a \geq k+3$ and $k \geq 1$, except for $(k,a)=(3,6)$ and $(4,7)$ (corresponding to the sporadic nontrivial solution $(6,7,10)$) and also if there exists a $t$ so that $a=t!-1$ and $k=t!-t$.  (It is not hard to argue that we may assume $a > k+2$.)  This last class of exceptions arises because $P_{t!-t,t!-1}(t!)=0$, and corresponds to class $1$ solutions.  If it is indeed the case that this polynomial is irreducible, then \cite{Serre03} implies there are primes $p$ for which $x^{\underline{k}}-a!$ has no root modulo $p$, so $a!b!=c!$ with $c=b+k$ also has no nontrivial solutions in integers $a$ and $b$.

\begin{conj}
    For integers $a \geq k+3$ and $k \geq 1$, if $x^{\underline{k}}-a!$ is reducible, then $(k,a)=(3,6)$, $(k,a)=(4,7)$, or there exists an integer $t$ so that $a=t!-1$ and $k=t!-t$.
\end{conj}

It is worth noting that even when $p$ is not one of the primes ruled out by Corollary \ref{cor:SerreSet}, the set of solutions to $(p-1)!b!=(b+k)!$, i.e., $a!b!=c!$ with $c=b+k$ and $a=p+1$, is still rather sparse.

\begin{prop}
    For $p$ prime, the equation
    $$
    (p-1)!b! = (b+k)!
    $$
    has at most $nk/p + 1$ solutions in the integers with $b \leq n$.
\end{prop}
\begin{proof}
    Suppose $(p-1)! b! = (b+k)!$.  Then, $(p-1)! = F_k(b)-1$.  Taking this equation modulo $p$ and applying Wilson's Theorem gives
    $$
    -1 \equiv F_k(b) - 1 \pmod{p},
    $$
    i.e., $F_k(b) \equiv 0 \pmod{p}$.  However, this equation has at most $k$ solutions $A_k \subseteq \mathbb{Z}/p\mathbb{Z}$.  If $b \pmod{p} \not \in A_k$, then $(p-1)!b! \neq (b+k)!$.  Since the set of $b \in [1,n]$ so that $b \pmod{p} \in A_k$ has cardinality $\lceil nk/p \rceil \leq nk/p + 1$, the conclusion follows.
\end{proof}

\section{Equidistribution and Falling Factorials} \label{sec:equidistribution}

In this section, we prove our main result: that if the quantities $\sqrt[k]{a!}$ are equidistributed $\pmod{1}$ in a very weak sense, then the fraction of integers $a$ for which there exist integers $b, c \geq 2$ with $a!b!=c!$ is $0$.  First, we present a technical lemma which will be useful in some subsequent bounds.

\begin{lemma} \label{inequalityLemma}
$(1-x)^{-1/x} \geq e(1+\frac{x}{2})$ if $x \in (0,1]$.
\end{lemma}
\begin{proof}
Let $f(x) = (1-x)^{-1/x}$, $g(x)=e(1+\frac{x}{2})$, and $h(x) = f(x) - g(x)$. Note that $\lim_{x \to 0^{+}} f(x) = e = g(0)$ by  L'H\^{o}pital's Rule, and $\lim_{x \to 1^{-}} f(x) \to \infty > g(1) = \frac{3}{2}e$. Thus, $\lim_{x\to0^{+}} h(x) = 0$ and $h(x)$ is positive as $x \to 1^+$.  To show that $h(x)$ is positive within $x \in (0,1)$, we take the derivative of $f$ and use the Taylor series substitutions $\log(1-x)= - \sum_{j \geq 1} x^j/j$ and $(1-x)^{-1} = \sum_{j \geq 0} x^j$:
 \begin{align*}
f'(x)= \frac{d}{dx}(1-x)^{-1/x} &= \frac{d}{dx} e^{-\log(1-x)/x} \\
&=\left (\frac{1}{x^2}\log(1-x) + \frac{1}{x(1-x)} \right )e^{-\log(1-x)/x} \\
&= \left ( \sum_{j=0}^\infty -\frac{x^{j-1}}{j+1} + \sum_{j=0}^\infty x^{j-1} \right )e^{-\log(1-x)/x} \\
&= e^{-\log(1-x)/x} \sum_{j=1}^\infty  \frac{x^{j-1} j}{j+1}
\end{align*}
Furthermore,
$$
    f''(x) =  e^{-\log(1-x)/x} \left ( \left ( \sum_{j=1}^\infty \frac{x^{j-1} j}{j+1} \right )^2 + \sum_{j=1}^\infty x^{j-2} \frac{j(j-1)}{j+1} \right ) > 0
$$
So $\lim_{x \rightarrow 0+} f'(x)=\frac{e}{2}  =g'(0)$. Since $g''(x)=0$ and $f''(x) > 0$, we conclude $h(x) > 0$ for all $x \in (0,1]$, thus $f(x) \geq g(x)$.
\end{proof}

Next, we provide careful asymptotics of the quantity $\sqrt[k]{r^{\underline{k}}}$ for $r$ large and $k$ fixed.

\begin{lemma}\label{lem:fallingfactapprox}
    For an integer $k \geq 2$ and any real $r \geq k$,
    $$
     0 \leq \left ( r - \frac{k-1}{2} \right ) - \sqrt[k]{r^{\underline{k}}} \leq \frac{k^2}{r-k}
    $$
\end{lemma}
\begin{proof}
First, note that by the AM-GM inequality,
$$
\sqrt[k]{r^{\underline{k}}} = \left ( r(r-1)\cdot(r-k+1) \right )^{1/k} \leq \frac{\sum_{j=0}^{k-1} (r-j)}{k} = r - \frac{1}{k} \sum_{j=0}^{k-1} j = r-\frac{k-1}{2}.
$$
Note that, for a positive integer $s$, we have
\begin{align*}
\int_{s-1}^s \log x \, dx  + \frac{1}{2} \left ( \log s - \log (s-1) \right ) &= s \log s - (s-1) \log (s-1) - 1  + \frac{1}{2} \left ( \log s - \log (s-1) \right ) \\
&= \left (s-\frac{1}{2} \right ) (\log s - \log (s-1) ) + \log s - 1 \\
&= \left (s-\frac{1}{2} \right ) \log \left ( \frac{s}{s-1} \right ) + \log s - 1 \\
&= \left (s-\frac{1}{2} \right ) \log \left ( 1 + \frac{1}{s-1} \right ) + \log s - 1 \\
&\leq \left (s-\frac{1}{2} \right ) \left ( \frac{1}{s-1} - \frac{1}{2(s-1)^2} + \frac{1}{3(s-1)^3} \right ) + \log s - 1 \\
&= \frac{1}{12(s-1)^2} + \frac{1}{6(s-1)^3} + \log s.
\end{align*}
Summing from $s=r-k+1$ to $s=r$ yields
\begin{align*}
\int_{r-k}^r \log x \, dx &\leq \sum_{j=0}^{k-1} \left [ \log (r-j) - \frac{1}{2} \left ( \log (r-j) - \log (r-j-1) \right ) + \frac{1}{12(r-j-1)^2} + \frac{1}{6(r-j-1)^3} \right ] \\
& \leq \left ( \sum_{j=0}^{k-1} \log (r-j) \right ) - \frac{1}{2} (\log r - \log(r-k)) + \frac{k}{12(r-k)^2} + \frac{k}{6(r-k)^3} 
\end{align*}
so that
\begin{align*}
\log \sqrt[k]{r^{\underline{k}}}    
&= \frac{1}{k} \sum_{j=0}^{k-1}\log(r-j) \\
&\geq \frac{1}{k}[x \log x - x]^{r}_{r-k} + \frac{1}{2k}\left [ \log (r) - \log (r-k) \right ] - \frac{1}{12} \cdot \frac{1}{(r-k)^2} - \frac{1}{6} \cdot \frac{1}{(r-k)^3} \\
&\geq \frac{1}{k} \left \{ [r \log r - r -((r-k) \log (r-k) -(r-k))] + \frac{1}{2} [\log r - \log (r-k)]\right \} - \frac{1}{4(r-k)^2} \\
&= \frac{1}{k} \left \{ r \log r -(r-k) \log (r-k) + \frac{1}{2} [\log r - \log(r-k) ]\right \} - 1 - \frac{1}{4(r-k)^2} \\
&= \frac{1}{k} \left \{ r (\log r - \log (r-k)) + k \log(r-k) + \frac{1}{2} [\log r - \log(r-k) ]\right \} - 1 - \frac{1}{4(r-k)^2}.
\end{align*}

Therefore,
\begin{align*}
\sqrt[k]{r^{\underline{k}}}     &\geq \left ( \frac{r-k}{r} \right )^{-r/k} \frac{r-k}{e} e^{-\frac{1}{4(r-k)^2}} \cdot \left ( \frac{r}{r-k} \right )^{\frac{1}{2k}} \\
&= \left ( 1 -\frac{k}{r} \right )^{-r/k} \left ( 1 - \frac{k}{r} \right )^{\frac{-1}{2k}} \frac{r-k}{e} \cdot e^{-\frac{1}{4(r-k)^2}} \\
&\geq e \left (1+\frac{k}{2r} \right ) \left ( e^{-k/r}  \right )^{-\frac{1}{2k}} \frac{r-k}{e} \cdot e^{-\frac{1}{4(r-k)^2}},
\end{align*}
where we have used the facts that $(1 - \delta)^{-1/\delta} \geq e (1 + \delta/2)$ by Lemma \ref{inequalityLemma}, taking $\delta = k/r$.  Continuing,
\begin{align*}
\sqrt[k]{r^{\underline{k}}} &\geq e^{\frac{1}{2r}-\frac{1}{4(r-k)^2}} \left (r - \frac{k}{2} - \frac{k^2}{2r} \right )  \\
&\geq \left ( 1 + \frac{1}{2r}-\frac{1}{4(r-k)^2} \right ) \left (r - \frac{k}{2} - \frac{k^2}{2r} \right )  \\
&= r - \frac{k-1}{2} + \frac{rk(4k^2 - 1) + 2k^3 - 2(2k^2+k+1)r^2}{8r^2(r-k)} \\
&\geq r - \frac{k-1}{2} - \frac{2k^2+k+1}{4(r-k)} \\
&\geq r - \frac{k-1}{2} - \frac{k^2}{(r-k)}
\end{align*}
where the last bound following by invoking the assumption that $k \geq 2$.
\end{proof}

Since integer solutions to $a!b!=c!$ correspond to roots of the polynomial $x^{\underline{k}}-a!$ with $k=c-b$, we can use the preceding lemma to describe a short interval into which $c$ must fall.
    
\begin{theorem} \label{thm:universalBoundOnC}
For $k \geq 2$, if $c > k$ is a root of $P(x)=x^{\underline{k}}-a!$, then
$$
c \in [\sqrt[k]{a!} + \frac{k-1}{2},\sqrt[k]{a!} + \frac{k-1}{2}+\frac{k^2}{c-k}]
$$    
\end{theorem}
\begin{proof}
If $P(c)=0$, then, by Lemma \ref{lem:fallingfactapprox},
\begin{align*}
    \left ( c - \frac{k-1}{2} - \frac{k^2}{c-k} \right )^k - {a!} \leq 0 = c^{\underline{k}} - {a!} \leq \left ( c - \frac{k-1}{2} \right )^k - {a!} 
\end{align*}
so 
$$
c\geq \sqrt[k]{a!} + \frac{k-1}{2} .
$$
and
$$
c \leq \sqrt[k]{a!} + \frac{k-1}{2} + \frac{k^2}{c-k}
$$
\end{proof}

The preceding theorem not only narrowly locates $c$ as a function of $a$ and $k$, but also does so for $\sqrt[k]{a!} \pmod{1}$ relative to $c$ because one endpoint is a half-integer and the other is nearly a half-integer.  Thus, we obtain the following main result.

\begin{theorem}
    Let $S' = \{ \sqrt[k]{a!} \mod 1 : 2 \leq a \leq A \; \& \; 2 \leq k \leq 5 \log \log A \}|$.  Suppose that when $I \subset [0,1]$ is an interval, $|S' \cap I| = |S'|(|I| + f(|S'|))$ for some function $f$.  Then the number of solutions to $a!b!=c!$ with $a < A$ and $k := c-b \geq 2$ is
    $$ 
    O(\sqrt{A} (\log \log(A))^3) + f(\sqrt{A}) A \log \log A.
    $$
    Furthermore, if $f(t) = o(1/\log \log t)$, then the number of $a < A$ for which there exists a solution to $a!b!=c!$ is $o(A)$.    
\end{theorem}
\begin{proof}
    By Erd\H{o}s \cite{Erdos91}, for sufficiently large $c$, we have $k \leq 5 \log \log c$.  Note that the interval 
    $$
    I_{c,k} := \left [ c - \frac{k-1}{2} - \frac{k^2}{c-k}, c- \frac{k-1}{2}\right ]
    $$ 
    has length $k^2/(c-k)$. Furthermore, modulo $1$, the interval $I'_k := I_k \pmod 1$ is $[1-k^2/(c-k),1]$ if $k$ is odd or $[1/2-k^2/(c-k),1/2]$ if $k$ is even.  These are two nested families of intervals, so in a collection $\{I_{c,k}\}_{(c,k) \in \mathcal{J}}$, the total length of each is at most $\max_{(c,k) \in \mathcal{J}} k^2/(c-k)$.

    We start by splitting $S := \{ \sqrt[k]{a!}  : 2 \leq a \leq A \; \& \; 2 \leq k \leq 5 \log \log A \}$ in two.  Let $1 \leq \gamma \leq A$, $S_1 = \{\sqrt[k]{a!} : 2 \leq a \leq \gamma \; \& \; k \leq 5 \log \log \gamma \}$, and $S_2 = S \setminus S_1$.  Also write $S' = \{s \pmod 1 : s \in S\}$ and $S_j' = \{s \pmod 1 : s \in S_j\}$ for $j=1,2$. Note that whenever we define $X' := \{x \pmod 1 : x \in X\}$ from a set $X$, we consider $X'$ a multiset when computing cardinalities.  Then let 
    $$
    \mathcal{I} := \bigcup_{k,c \geq 2} I_{c,k}
    $$
    and
    $$
    \mathcal{I}': = \bigcup_{k,c \geq 2} I'_{c,k}
    $$
    The number of elements of $S$ that fall into $\mathcal{I}$ is given by
    $$
    |S \cap \mathcal{I}|  \leq |S_1|+ |S_2 \cap \mathcal{I} | \leq |S_1| + |S_2' \cap \mathcal{I}'|
    $$
    For the first summand, $|S_1| \leq 5 \gamma \log\log \gamma$. For the second summand, note that $S'_2 \cap \mathcal{I}' \subseteq \mathcal{I}'_1 \cup \mathcal{I}'_2$, where $\mathcal{I}'_j = I'_{c_j,k_j}$ with $j=1,2$ and $(c_j,k_j)$ is chosen to be the pair maximizing $k_j^2/(c_j-k_j)$ for $k_j \equiv j \pmod{2}$ such that $S'_2 \cap I'_{c_j,k_j} \neq \emptyset$. If $\sqrt[k]{a!} \in S_2$, by Theorem \ref{thm:universalBoundOnC}, $c \geq \sqrt[k]{a!} \geq \ \gamma$. So
    $$
    \frac{k^2}{c-k} \geq \frac{(5 \log \log \gamma)^2}{\gamma-5 \log \log \gamma} \geq \frac{50 (\log \log \gamma)^2}{\gamma}
    $$
    for sufficiently large $\gamma$, and $\mathcal{I}'_j := \mathcal{I}_j \pmod 1 \subseteq I_{\gamma,5 \log\log \gamma}'$ for $j=1,2$. Therefore, 
    \begin{align*}
        |S_2' \cap \mathcal{I}'| &\leq \left ( \frac{50 (\log\log \gamma)^2}{\gamma} + f(\gamma)\right) \cdot |S_2|  \\
        &= \left ( O \left (\frac{(\log\log\gamma)^2}{\gamma} \right ) + f(\gamma) \right ) \cdot A \log\log A
    \end{align*}
    
    Putting everything together and taking $\gamma = \sqrt{A}$, we have
    \begin{align*}
        |S \cap \mathcal{I}|  &\leq |S_1|+ |S_2 \cap \bigcup_{c,k} I_{c,k} | \\
        &= O(\gamma\log\log\gamma) + \left ( O \left (\frac{(\log\log\gamma)^2}{\gamma} \right ) + f(\gamma) \right  ) \cdot A \log\log A  \\
        &= O(\sqrt{A} (\log \log(A))^3) + f(\sqrt{A}) A \log \log A.
    \end{align*}
\end{proof}

Thus, we conclude with the following conjecture which implies the appropriate equidistribution needed to conclude that the set of $a$ which participate in solutions to $a!b!=c!$ is indeed sparse.

\begin{conj}
Let $S = \{ \sqrt[k]{a!} \mod 1 : 2 \leq a \leq A \; \& \; 2 \leq k \leq 5 \log \log A \}|$.  Then, for any interval $I \subset [0,1]$, 
$$
|S \cap I| = |S| |I| + O(|S|^{1-\epsilon})
$$
for some positive $\epsilon > 0$ as $A \rightarrow \infty$.
\end{conj}

\section{Acknowledgments}

The authors would like to thank Michael Filaseta and Igor Shparlinski for helpful discussions.

\bibliographystyle{plain} 
\bibliography{refs}

\end{document}